\documentclass{amsart}
\usepackage{amssymb,graphpap}
\usepackage{enumitem}
\newtheorem{theorem}{Theorem}[section]
\newtheorem{proposition}[theorem]{Proposition}

\theoremstyle{definition}
\newtheorem{definition}[theorem]{Definition}
\newtheorem{notation}[theorem]{Notation}

\theoremstyle{remark}
\newtheorem{remark}[theorem]{Remark}

\def\la{\langle}
\def\ra{\rangle}
\def\id{\text{{id}}}

\def\cA{\mathcal{A}}

\def\cP{\mathcal{P}}
\def\cB{\mathcal{B}}

\def\CC{{\mathbb{C}}}
\def\NN{{\mathbb{N}}}

\def\ff{\varphi}

\def\ii{{\bf i}}
\def\jj{{\bf j}}
\def\kk{\kappa}
\def\tail{\text{tail}}
\def\cBX{\cB\la X\ra}
\def\cAt{\cA_\tail}
\def\cAtX{\cAt\la X\ra}
\begin{document}
\date{July 3, 2008}
\sloppy

\title[Noncommutative de Finetti Theorem]{A Noncommutative de Finetti Theorem:\\
Invariance under Quantum Permutations\\ is Equivalent to\\ Freeness with Amalgamation}
\author[C. K\"ostler]{Claus K\"ostler}
\address{University of Illinois at Urbana-Champaign, Department of Mathematics, Altgeld
Hall, 1409 West Green Street, Urbana, 61801, USA} \email{koestler@uiuc.edu}

\author[R. Speicher]{Roland Speicher $^{(\dagger)}$}
\thanks{$^\dagger$
Research supported by Discovery and LSI grants from NSERC (Canada) and by a Killam
Fellowship from the Canada Council for the Arts}
\address{Queen's University, Department of Mathematics and Statistics,
Jeffery Hall, Kingston, ON K7L 3N6, Canada}
\email{speicher@mast.queensu.ca}

\begin{abstract}
We show that the classical de Finetti theorem has a canonical noncommutative counterpart
if we strengthen ``exchangeability'' (i.e., invariance of the joint distribution of the
random variables under the action of the permutation group) to invariance under the
action of the quantum permutation group. More precisely, for an infinite sequence of
noncommutative random variables $(x_i)_{i\in\NN}$, we prove that invariance of the joint
distribution of the $x_i$'s under quantum permutations is equivalent to the fact that the
$x_i$'s are identically distributed and free with respect to the conditional expectation
onto the tail algebra of the $x_i$'s.
\end{abstract}

\maketitle

\section{Introduction}
The de Finetti theorem states that an infinite family of random variables whose
distribution is invariant under finite permutations (such a family is called
\emph{exchangeable}) is independent and identically distributed with respect to the
conditional expectation onto the tail algebra of the random variables. Since the
implication in the other direction is fairly elementary one has the equivalence between
exchangeability and conditional independence. See, e.g., \cite{Kallenberg} for an
exposition on the classical de Finetti theorem.

In a noncommutative context classical random variables are replaced by, typically
noncommuting, operators on Hilbert spaces. The expectation with respect to a probability
measure is then replaced by a state on the algebra generated by these operators. The
notion of exchangeability makes of course also sense in such a context, as invariance of
mixed moments under permutations of the random variables, and one can ask what
exchangeability implies in such a more general context. It turns out that in the
noncommutative world there are actually many quite different possibilities for
exchangeable random variables. It was shown in \cite{Claus:factorization} that they all
possess some kind of factorization property; but, as one sees from the variety of
examples, one cannot expect that exchangeability implies some fixed kind of independence.
Indeed, both independence and freeness provide basic examples for exchangeable random
variables. (See also \cite{Lehner, Claus:Lehner} for more on this.)

However, if one moves into the noncommutative realm, one should also take into account
that invariance under permutations is a commutative concept and should be replaced by its
noncommutative analogue. To provide such noncommutative analogues of actions of groups
was one of the motivations for the creation of the theory of quantum groups, which has
been developed very extensively within the last 20 years or so. In particular, Wang
introduced in \cite{Wang} the noncommutative analogue of the permutation group $S_n$,
namely the \emph{quantum permutation group $A_s(n)$}. So if one considers noncommuting
random variables, it is natural to replace the requirement of invariance under
permutations by the stronger requirement of invariance under quantum permutations.
Classical (commuting) independent random variables do not satisfy this stronger form of
exchangeability any more and, as we will show in our main theorem, this noncommutative
version of exchangeability singles out again a very special situation - namely freeness
with amalgamation. In the same way as classical exchangeability is equivalent to
conditional independence, quantum exchangeability is equivalent to freeness with
amalgamation.

Thus our noncommutative de Finetti theorem is another instance of the general philosophy
that freeness plays in the noncommutative world the same role as independence plays in
the commutative world. Note that freeness is not a hidden assumption in our de Finetti
theorem, but it is a consequence of replacing the commutative permutation group by its
noncommutative counterpart.

Here is the statement of our noncommutative de Finetti theorem. All relevant notions will
be defined in Sections 2 and 4.

\begin{theorem}\label{thm:deFinetti}
Let $(\cA,\ff)$ be a $W^*$-probability space and consider an infinite sequence
$(x_i)_{i\in\NN}$ in $\cA$. Then the following two statements are equivalent:
\begin{enumerate}[label=\textnormal{(\alph*)}]
\item \label{item:deFinetti-a}
The joint distribution of $(x_i)_{i\in\NN}$ with respect to $\ff$ is invariant under
quantum permutations.
\item \label{item:deFinetti-b}
The sequence $(x_i)_{i\in\NN}$ is identically distributed and free with respect to the
conditional expectation $E$ onto the tail algebra of the $(x_i)_{i\in\NN}$.
\end{enumerate}

\end{theorem}
To be precise, $E$ denotes the $\ff_\infty$-preserving conditional expectation from
$\cA_\infty$ onto the tail algebra $\cA_\tail$ of $(x_i)_{i \in \NN}$, where $\cA_\infty$
is the von Neumann subalgebra generated by $(x_i)_{i\in\NN}$ and $\ff_\infty$ is the
restriction of $\ff$ to $\cA_\infty$. We want to point out that, if $\ff$ is a trace,
then $E$ can be chosen to be the $\ff$-preserving conditional expectation from $\cA$ onto
$\cA_\tail$. More care is needed in the general case of a non-tracial state. Here we need
to ensure the existence of the conditional expectation $E$. As we will show in Section 4,
this can always be achieved for exchangeable random variables, after restriction to the
W*-probability space $(\cA_\infty,\ff_\infty)$.

 Our paper is organized as follows. In the next section we collect the preliminaries.
On one side, we present the definition of the quantum permutation group and the notion of
invariance under quantum permutations. On the other side, we recall the basic definitions and
relevant results about free independence with amalgamation. In Section 3, we will prove the
``easy'' implication of our de Finetti theorem, namely that freeness with amalgamation implies
invariance under quantum permutations. This is actually not as elementary as in the classical
case (where it follows directly from the fact that independence is a rule for expressing mixed
moments in terms of moments of the single random variables) and we will have to use some of the
basic theory of freeness for this proof. In Section 4, we will define the tail algebra of our
sequence of random variables, and show some basic properties of the corresponding conditional
expectation. Section 5 will finally give the proof of the other implication of our de Finetti
theorem, Theorem \ref{thm:deFinetti}. The paper closes with an example which shows that, as in
the classical case, one needs infinitely many random variables in our de Finetti theorem:
quantum exchangeability of \emph{finitely} many random variables does not necessarily imply
freeness with amalgamation.

\section{Preliminaries}

\subsection{Noncommutative probability spaces and distributions of random variables}
Here we recall the basic notions of non-commutative probability spaces and distributions of
random variables; this is just to have a convenient language for our main statements.
\begin{definition}
1) A \emph{noncommutative probability space} $(\cA,\ff)$ consists of a unital algebra
$\cA$ and a unital linear functional $\ff$.\\
2) A \emph{$W^*$-probability space} $(\cA,\ff)$ is a von Neumann algebra $\cA$ together with a
faithful normal state $\ff$ on $\cA$.
\end{definition}

Note that for a $W^*$-probability space we do not require that our state $\ff$ is a trace.

\begin{definition}
Let $(\cA,\ff)$ be a non-commutative probability space and $(x_i)_{\in\in\NN}$ a sequence
in $\cA$. The \emph{joint distribution of $(x_i)_{i\in\NN}$} is given by the collection
of all moments $\ff(x_{i(1)}\cdots x_{i(n)})$ for all $n\in\NN$ and all
$i(1),\dots,i(n)\in\NN$.
\end{definition}

\subsection{Quantum Permutation Group}
Wang introduced in \cite{Wang} the following noncommutative version of the permutation
group $S_n$.

\begin{definition}\label{def:quantumpermutation}
The \emph{quantum permutation group} $A_s(n)$ is defined as the universal unital
$C^*$-algebra generated by elements $u_{ij}$ ($i,j=1,\dots,n$) such that we have
\begin{itemize}
\item
each $u_{ij}$ is an orthogonal projection: $u_{ij}^*=u_{ij}=u_{ij}^2$ for all
$i,j=1,\dots,n$
\item
the elements in each row and column of $u=(u_{ij})_{i,j=1}^n$ form a partition of unity,
i.e., are orthogonal and sum up to 1: for each $i=1,\dots,n$ and $k\not=l$ we have
$$u_{ik}u_{il}=0\qquad\text{and} \qquad u_{ki}u_{li}=0;$$
and for each $i=1,\dots,n$ we have
$$\sum_{k=1}^n u_{ik}=1=\sum_{k=1}^n u_{ki}.$$
\end{itemize}
\end{definition}

Note that the above requirements imply in particular that the matrix
$u=(u_{ij})_{i,j=1}^n$ is orthogonal, i.e., for each $i,j=1,\dots, n$ we have
$$\sum_{k=1}^n u_{ik}u_{jk}=\delta_{ij}1\qquad\text{and}\qquad
\sum_{k=1}^n u_{ki}u_{kj}=\delta_{ij}1.
$$

$A_s(n)$ is a compact quantum group in the sense of Woronowicz \cite{Wor}. That this is
the right noncommutative version of the permutation group can be seen from the fact that
adding commutativity of the $u_{ij}$ to the above definition yields the group algebra of
the permutation group and that, by a theorem of Wang \cite{Wang}, $A_s(n)$ is the biggest
Hopf algebra coacting on a space of $n$ points. For more information on $A_s(n)$, see
\cite{BC,BBC}.

For $n=1,2,3$ the quantum permutation group is the same as the usual permutation group,
i.e., in these cases $A_s(n)$ is isomorphic as a Hopf algebra to  $\CC S_n$. For $n\geq
4$, however, the quantum version is strictly larger than the classical one; this can be
seen, for example, by finding representations of the $u_{ij}$ which do not commute. Here
is such a representation in the case $n=4$:
$$u=\begin{pmatrix}
q_1&1-q_1&0&0\\
1-q_1&q_1&0&0\\
0&0&q_2&1-q_2\\
0&0&1-q_2&q_2
\end{pmatrix},$$
where $q_1$ and $q_2$ are arbitrary projections. If we take them non-commuting, then the
$C^*$-algebra generated by $q_1$ and $q_2$, which is a quotient of $A_s(4)$, is infinite
dimensional.

\begin{definition}
Consider a noncommutative probability space $(\cA,\ff)$ and a sequence of random
variables $(x_i)_{i\in\NN}$ in $\cA$. We say that the joint distribution (with respect to
$\ff$) of this sequence is \emph{invariant under quantum permutations} or that the
sequence is \emph{quantum exchangeable} if, for any $k\in\NN$, the natural action of
$A_s(k)$ on the $k$-tuple $(x_1,\dots,x_k)$, given by
$$x_i\mapsto \tilde x_i:=\sum_{j=1}^k u_{ij}\otimes x_j \in A_s(k)\otimes \cA,$$
does not change the distribution, i.e., the joint distribution of the $k$-tuple
$(x_1,\dots,x_k)$ with respect to $\ff$ is the same as the joint distribution of the $k$-tuple
$(\tilde x_1,\dots,\tilde x_k)$ with respect to $\id\otimes \ff$.

More explicitly, this means: for all $k,n\in\NN$ and all $1\leq i(1),\dots,i(n)\leq k$ we
have
\begin{equation}\label{eq:invariance}
\ff(x_{i(1)}\cdots x_{i(n)})=\sum_{j(1),\dots,j(n)=1}^k u_{i(1)j(1)}\cdots
u_{i(n)j(n)}\cdot \ff(x_{j(1)}\cdots x_{j(n)})
\end{equation}
as an equality in $A_s(k)$.
\end{definition}

To say it in other words, invariance under quantum permutations asks for the validity of
\eqref{eq:invariance} for any matrix $u=(u_{i,j})_{i,j=1}^k$ whose entries are bounded
operators on some Hilbert space and satisfy the defining relations of $A_s(k)$ from
Definition \ref{def:quantumpermutation}. Note that we do not apply a state on the
elements from $A_s(k)$ to get equality in \eqref{eq:invariance}, but ask for an algebraic
identity in $A_s(k)$.

For a permutation $\sigma\in S_k$ the permutation matrix $(e_{ij})_{i,j=1}^k$ with
$e_{ij}=\delta_{\sigma(i)j}$ provides an example of such a $u$, in this case
\eqref{eq:invariance} gives
\begin{align*}
\ff(x_{i(1)}\cdots x_{i(n)})&=\sum_{j(1),\dots,j(n)=1}^k \delta_{\sigma(i(1))j(1)}\cdots
\delta_{\sigma(i(n))j(n)} \ff(x_{j(1)}\cdots x_{j(n)})\\&=\ff(x_{\sigma(i(1))}\cdots
x_{\sigma(i(n))}),
\end{align*}
which is just the invariance of the distribution of $(x_i)_{i\in\NN}$ under the
permutation $\sigma$. Thus invariance under quantum permutations includes in particular
invariance under permutations; quantum exchangeable random variables are in particular
exchangeable.

\subsection{Freeness with Amalgamation}

Here we collect the basic definitions and needed facts about freeness. For general
introductions on free probability theory, see \cite{VDN,NS,HP}. In the classical de
Finetti theorem we do not get ordinary independence of the random variables, but have to
condition this over the tail algebra. In the same spirit, in our noncommutative de
Finetti theorem, we cannot hope for ordinary freeness with respect to the state $\ff$,
but must expect that we have to condition this with respect to the tail algebra of the
random variables. Voiculescu introduced such a conditional version of freeness (called
operator-valued freeness or freeness with amalgamation) from the very beginning and
developed its basic theory in \cite{Voi:operator}. In \cite{Speicher:Memoir} this concept
was treated from the combinatorial point of view and it was shown that the theory of free
cumulants extends to the operator-valued frame. As our proof of the ``easy'' direction of
theorem \eqref{thm:deFinetti} relies on free cumulants, we will below recall the relevant
facts about operator-valued free cumulants.

Let us first give the definition of an operator-valued probability space and freeness.
This will be done in a general, algebraic context, as one implication of our de Finetti
theorem does only require such general structure.

Recall that a conditional expectation $E:\cA\to\cB$ (for unital algebras $\cB\subset\cA$)
is a linear map which satisfies $E[b]=b$ for all $b\in\cB$ and the bimodule property
\begin{equation*}
E[b_1 a b_2]=b_1 E[a] b_2\qquad \text{for all $b_1,b_2\in\cB$ and for all $a\in\cA$.}
\end{equation*}

\begin{definition}
1) An \emph{operator-valued probability space} $(\cA, E:\cA\to\cB)$ consists of a unital
algebra $\cA$, a unital subalgebra $\cB\subset \cA$ and a conditional expectation
$E:\cA\to\cB$. Elements in $\cA$ are called \emph{(operator-valued) random variables}.

2) For a unital algebra $\cB$ we denote by $\cB\la X\ra$ the \emph{$\cB$-valued polynomials} in
the formal variable $X$; these are linear combinations of elements of the form $b_0Xb_1X\cdots
b_{n-1}Xb_{n}$ for all $n=0,1,2,\dots$ and all $b_0,\dots,b_n\in\cB$. (For $n=0$, this is just
$b_0$.) Elements from $\cB$ do not commute with $X$ (with the exception of $1\cdot X=X=X\cdot
1$). For $p\in\cB\la X\ra$ and $a\in\cA$ (for some algebra $\cA$ which contains $\cB$ as a
subalgebra) we denote by $p(a)\in\cA$ the element which one gets by replacing the variable $X$
by $a$.

3) Let $(\cA,E:\cA\to\cB)$ be an operator-valued probability space and $(x_i)_{i\in\NN}$ a
sequence of random variables in $\cA$. We say that the sequence is \emph{identically
distributed} (with respect to $E$) if for each $p\in \cB\la X\ra$ the expression $E[p(x_i)]$
does not depend on $i\in\NN$.
\end{definition}

In the case of an ordinary noncommutative probability space, i.e., $\cB=\CC$ and $E=\ff$,
the $b_i$ in the definition of $\cB\la X\ra=\CC\la X\ra$ are superfluous and $\CC\la
X\ra$ are just ordinary polynomials; in this case ``identically distributed'' just means
that for each $n\in\NN$ the ordinary moment $\ff(x_i^n)$ does not depend on $i$.

\begin{definition}
Let $(\cA,E:\cA\to\cB)$ be an operator-valued probability space and $I$ an arbitrary
index set. Random variables $(a_i)_{i\in I}$ are called \emph{free with respect to $E$}
(or \emph{free with amalgamation over $\cB$}) if we have for all $n\in\NN$, all
$i(1),\dots,i(n)\in I$ with $i(1)\not=i(2)\not=\dots\not= i(n)$ and all $\cB$-valued
polynomials $p_1,\dots,p_n\in \cBX$ with $E[p_m(a_{i(m)})]=0$ ($m=1,\dots,n$) that also
$$E[p_1(a_{i(1)})\cdots p_n(a_{i(n)})]=0.$$
\end{definition}

The special case where $\cB$ is $\CC$ (and thus $E$ a unital linear functional
$\ff:\cA\to\CC$) gives just the usual definition of freeness.

\subsection{Operator-valued free cumulants}

The combinatorial theory of operator-valued freeness \cite{Speicher:Memoir} relies on the
notions of non-crossing partitions and free cumulants. We will now recall these notions.

\begin{definition}
1) A \emph{partition} $\pi$ of a set $S$ is a decomposition $\pi=\{V_1,\dots,V_r\}$ of
$S$ into disjoint, non-empty subsets $V_i$. The elements $V_i$ are called the
\emph{blocks} of $\pi$. We denote the partitions of $S$ by $\cP(S)$. In the case
$S=\{1,\dots,n\}$, we just write $\cP(n)$.

2) For $\pi,\sigma\in\cP(n)$ we say that $\pi\leq\sigma$ if each block of $\pi$ is
contained in a block of $\sigma$.

2) Let $S$ be an ordered set. A partition $\pi\in\cP(S)$ is called \emph{non-crossing} if
there do not exist two different blocks $V,W$ of $\pi$ such that we have
$s_1<t_1<s_2<t_2$ and $s_1,s_2\in V$ and $t_1,t_2\in W$. The set of non-crossing
partitions of $S$ is denoted by $NC(S)$, or just $NC(n)$ in the case of
$S=\{1,\dots,n\}$.
\end{definition}

If one draws partitions by connecting elements belonging to the same block by
half-circles below the numbers $1,\dots,n$, then the partition is non-crossing if and
only if one does not get crossings between different blocks in such a drawing. Another
characterization of a non-crossing partition is the following recursive description:
$\pi\in\cP(S)$ is non-crossing if at least one of the blocks of $\pi$, say $V$, is an
interval (i.e., consists of consecutive numbers) and if $\pi\backslash V$ is a
non-crossing partition of $S\backslash V$.

\begin{definition}
Let $(\cA,E:\cA\to\cB)$ be an operator-valued probability space.

1) A map $\rho:\cA^n\to\cB$ (for $n\in\NN$) is called a \emph{$\cB$-functional} if it is
$n$-linear and if we have for all $b_0,\dots,b_n\in\cB$ and all $a_1,\dots,a_n\in\cA$
that
$$\rho(b_0a_1b_1,a_2b_2,\dots,a_{n-1}b_{n-1},a_nb_n)=b_0\rho(a_1,b_1a_2,\dots,
b_{n-2}a_{n-1},b_{n-1}a_n)b_n.$$

2) Let, for each $k\in\NN$, a $\cB$-functional $\rho_k:\cA^k\to\cB$ be given. Then, for
$n\in\NN$ and $\pi\in NC(n)$ we define a $\cB$-functional $\rho_\pi:\cA^n\to\cB$
recursively as follows. If $\pi$ is the maximal element $1_n\in NC(n)$, which has only
one block, then we put for all $a_1,\dots,a_n\in\cA$
$$\rho_{1_n}[a_1,\dots,a_n]=\rho_n(a_1,\dots,a_n).$$
Otherwise, let $V=(i+1,\dots,i+r)$ be an interval of $\pi$. Then, for
$a_1,\dots,a_n\in\cA$,
$$\rho_\pi[a_1,\dots,a_n]=\rho_{\pi\backslash
V}[a_1,\dots,a_{i-1},a_i\cdot \rho_r(a_{i+1},\dots,a_{i+r}),a_{i+r+1},\dots,a_n]$$
\end{definition}

As illustration of this definition consider
$$\pi=\bigl\{\{1,10\},\{2,5,9\},\{ 3,4 \} , \{ 6 \} , \{ 7,8 \} \bigr\}\in NC(10),$$

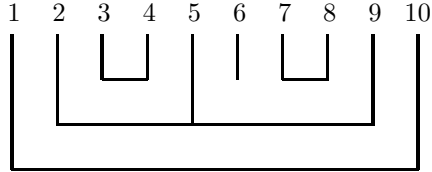
\setlength{\unitlength}{0.6cm} $$\begin{picture}(9,4)\thicklines \put(0,0){\line(0,1){3}}
\put(0,0){\line(1,0){9}} \put(9,0){\line(0,1){3}} \put(1,1){\line(0,1){2}}
\put(1,1){\line(1,0){7}} \put(4,1){\line(0,1){2}} \put(8,1){\line(0,1){2}}
\put(2,2){\line(0,1){1}} \put(2,2){\line(1,0){1}} \put(3,2){\line(0,1){1}}
\put(5,2){\line(0,1){1}} \put(6,2){\line(0,1){1}} \put(6,2){\line(1,0){1}}
\put(7,2){\line(0,1){1}} \put(-0.1,3.3){1} \put(0.9,3.3){2} \put(1.9,3.3){3}
\put(2.9,3.3){4} \put(3.9,3.3){5} \put(4.9,3.3){6} \put(5.9,3.3){7} \put(6.9,3.3){8}
\put(7.9,3.3){9} \put(8.7,3.3){10}
\end{picture}$$
The corresponding $\rho_\pi$ is
$$\rho_\pi[a_1,\dots,a_{10}]=\rho_2\Bigl(a_1\cdot \rho_3\bigl(a_2\cdot \rho_2(a_3,a_4),a_5
\cdot\rho_1(a_6)\cdot \rho_2(a_7,a_8),a_9\bigr),a_{10}\Bigr).$$

\begin{definition}
Let $(\cA,E:\cA\to\cB)$ be an operator-valued probability space. The corresponding
\emph{operator-valued free cumulants} $(\kk_n^E)_{n\in\NN}$ are defined recursively by
the \emph{moment-cumulant formulas}: for each $n\in\NN$ and all $a_1,\dots,a_n\in\cA$ we
have
\begin{equation} \label{eq:moment-cumulant}
E[a_1\cdots a_n]=\sum_{\pi\in NC(n)} \kk_\pi^E[a_1,\dots,a_n].
\end{equation}
\end{definition}

Note that in the moment-cumulant formula \eqref{eq:moment-cumulant} the right hand side
is of the form $\kk_n^E(a_1,\dots,a_n)$ plus products of lower order terms; thus this can
indeed recursively be solved for the $\kk_n^E$. There is a quite a lot one can say about
the structure of the formulas for the $\kk_n^E$, but we will not need this here and refer
for more information on this to \cite{NS,Speicher:Memoir}. Here are as examples just the
first three cumulants:
$$\kk_1^E(a_1)=E[a_1],\qquad \kk_2^E(a_1,a_2)=E[a_1a_2]-E[a_1]\cdot E[a_2]$$
and
\begin{multline*}
\kk_3^E(a_1,a_2,a_3)= E[a_1a_2a_3]-E[a_1]\cdot E[a_2a_3]-E\bigl[a_1\cdot E[a_2]\cdot
a_3\bigr]\\- E[a_1a_2]\cdot E[a_3]+2 E[a_1]\cdot E[a_2]\cdot E[a_3].
\end{multline*}

The main result which we will use about free cumulants is that they characterize freeness
via the property ``vanishing of mixed cumulants''.

\begin{theorem}[\cite{Speicher:Memoir}]
Let $(\cA,E:\cA\to\cB)$ be an operator-valued probability space and consider, for some
index set $I$, random variables $(a_i)_{i\in I}$. Then the following are equivalent:
\begin{enumerate}
\item
The random variables $(a_i)_{i\in I}$ are free with respect to $E$.
\item
We have the vanishing of mixed operator-valued free cumulants: For all $n\geq 2$, all
$i(1),\dots,i(n)\in I$, and all $b_1,\dots,b_{n-1}\in\cB$ we have
$$\kk_n^E(a_{i(1)}b_1,\dots,a_{i(n-1)} b_{n-1},a_{i(n)})=0$$
whenever there are $1\leq k,l\leq n$ such that $i(k)\not= i(l)$.
\end{enumerate}
\end{theorem}

If we transfer this characterization from the $\kk_n^E$ to $\kk_\pi^E$ then freeness of
the $a_i$ implies that $\kk_\pi^E[a_{i(1)},\dots,a_{i(n)}]$ can only be non-zero when all
the $i$-indices belonging to the same block are equal. It will be convenient to have a
notation at hand which encodes that information.

\begin{notation}
For $n\in \NN$ and an $n$-tuple $\ii=(i(1),\dots,i(n))$ we denote by $\ker \ii\in\cP(n)$
that partition of $1,\dots,n$ which is determined by
$$\text{$k$ and $l$ are in the same block} \qquad\Leftrightarrow \qquad i(k)=i(l).$$
\end{notation}

With this notation we have: if $(a_i)_{i\in I}$ are free with respect to $E$, then
$\kk_\pi^E[a_{i(1)},\dots,a_{i(n)}]$ can only be non-zero for $\ker\ii\geq \pi$. Note
that $\ker \ii$ is in general a possibly crossing partition.

\section{Operator-valued free random variables\\ are invariant under Quantum Permutations}

We will now first prove the ``easy'' direction of our de Finetti theorem, namely that
random variables which are free with respect to a conditional expectation $E$ are
invariant under quantum permutations with respect to any $\ff$ which is compatible with
$E$. In contrast to the other direction this can be done in a purely algebraic frame,
thus we will treat this implication in the context of an arbitrary non-commutative
probability space. Note also that this implication does actually not require that our
sequence is infinite. This will only be crucial for the other implication.

\begin{proposition}
Let $(\cA,\ff)$ be a noncommutative probability space, $\cB\subset\cA$ a unital
subalgebra, and $E:\cA\to\cB$ a conditional expectation such that $\ff=\ff\circ E$.
Consider a sequence $(x_i)_{i\in\NN}$ in $\cA$ which is identically distributed and free
with respect to $E$. Then the joint distribution of the sequence $(x_i)_{i\in\NN}$ with
respect to $\ff$ is invariant under quantum permutations.
\end{proposition}

\begin{proof}
Fix $n,k$ and $\ii=(i(1),\dots,i(n))$ with $1\leq i(1),\dots,i(n)\leq k$. We have
\begin{align*}
&\sum_{j(1),\dots,j(n)=1}^k u_{i(1)j(1)}\cdots u_{i(n)j(n)}\cdot \ff\bigl(
x_{j(1)}\cdots x_{j(n)}\bigr)\\
&= \sum_{j(1),\dots,j(n)=1}^k u_{i(1)j(1)}\cdots u_{i(n)j(n)}\cdot
\ff\bigl(E[x_{j(1)}\cdots
x_{j(n)}]\bigr)\\
&=\sum_{j(1),\dots,j(n)=1}^k u_{i(1)j(1)}\cdots u_{i(n)j(n)} \cdot\ff\bigl(\sum_{\pi\in
NC(n)} \kk_\pi^E[x_{j(1)},\dots, x_{j(n)}]\bigr)\\
&=\sum_{\pi\in NC(n)}\sum_{j(1),\dots,j(n)=1}^k u_{i(1)j(1)}\cdots u_{i(n)j(n)}
\cdot\ff\bigl( \kk_\pi^E[x_{j(1)},\dots, x_{j(n)}]\bigr).
\end{align*}
Now we note that because of the vanishing of mixed cumulants for free variables the term
$\kk_\pi^E[x_{j(1)},\dots, x_{j(n)}]$ is only non-vanishing if $\ker \jj\geq\pi$, where
$\jj=(j(1),\dots,j(n))$. Furthermore, by the identical distribution with respect to $E$
of our random variables, for any $\jj$ with $\ker \jj\geq\pi$ the term
$\kk_\pi^E[x_{j(1)},\dots, x_{j(n)}]$ has the same value, which we denote by $\kk_\pi^E$.
Thus we can continue the above calculation as follows:
\begin{multline*}
\sum_{j(1),\dots,j(n)=1}^k u_{i(1)j(1)}\cdots u_{i(n)j(n)}\cdot \ff(x_{j(1)}\cdots x_{j(n)})\\
=\sum_{\pi\in NC(n)}\ff\bigl( \kk_\pi^E\bigr) \sum_{\substack{{j(1),\dots,j(n)=1,\dots,k}\\{\ker
\jj \; \geq \pi}}} u_{i(1)j(1)}\cdots u_{i(n)j(n)}.
\end{multline*}
The sum over $j(1),\dots,j(n)$ with $\ker \jj\geq \pi$ means that we sum for each block
of $\pi$ independently over one $j$-variable. Since $\pi$ is non-crossing at least one of
its blocks is an interval, i.e., of the form $\{p,p+1,p+2,\dots,p+s\}$ for some $1\leq
p\leq p+s\leq k$. Then we have $j(p)=j(p+1)=\cdots=j(p+s)$, the sum over this variable is
independent of the other sums; and it only involves
$$\sum_{j=1}^k u_{i(p)j}u_{i(p+1)j}\cdots u_{i(p+s)j}.$$
Because of the orthogonality of different elements in the same row of
$u=(u_{ij})_{i,j=1}^k$, the term $u_{i(p)j}u_{i(p+1)j}\cdots u_{i(p+s)j}$ is zero for any
$j$ unless $i(p)=i(p+1)=\dots=i(p+s)$. In the latter case, $u_{i(p)j}u_{i(p+1)j}\cdots
u_{i(p+s)j}=u_{i(p)j}$ and the sum over $j$ just gives $1$. In this way we are left with
the same problem as before but with the positions $p,p+1,\dots,p+s$ removed. For $\pi$ we
have just removed one of its interval blocks. Since $\pi$ is non-crossing, we can now
find another interval block in the new partition and repeat the above argument. In this
way we can do all the summations over the blocks of $\pi$ in an inductive way. In each
step the $i$-indices must agree on the considered block of $\pi$ to get a non-vanishing
contribution. If they do then the summation over the $j$-index for this block gives 1. So
we get in the end that
$$\sum_{\substack{{j(1),\dots,j(n)=1,\dots,k}\\{\ker \jj\;\geq \pi}}}
u_{i(1)j(1)}\cdots u_{i(n)j(n)}=\begin{cases} 1,&\ker \ii\geq\pi\\
0,& \text{otherwise}
\end{cases}.
$$
Thus, by recalling that $\kk_\pi^E$ is equal to $\kk_\pi^E[x_{i(1)},\dots, x_{i(n)}]$ for
any $i$ with $\ker \ii\geq \pi$,
 we have
\begin{alignat*}{1}
\sum_{j(1),\dots,j(n)=1}^k u_{i(1)j(1)}\cdots u_{i(n)j(n)}&\cdot \ff(x_{j(1)}\cdots x_{j(n)})\\
&=\sum_{\substack{{\pi\in NC(n)}\\{\ker \ii\;\geq\pi}}}\ff\bigl( \kk_\pi^E\bigr)
=\ff\bigl(\sum_{\substack{{\pi\in NC(n)}\\{\ker \ii\;\geq\pi}}} \kk_\pi^E\bigr)\\
&=\ff\bigl(\sum_{\substack{{\pi\in NC(n)}\\{\ker \ii\;\geq\pi}}} \kk_\pi^E[x_{i(1)},\dots, x_{i(n)}]\bigr)\\
&=\ff\bigl(E[x_{i(1)}\cdots x_{i(n)}]\bigr)
=\ff\bigl(x_{i(1)}\cdots x_{i(n)}\bigr).
\end{alignat*}
\end{proof}

\section{Properties of the conditional expectation\\ onto the tail algebra}

In order to make the step from quantum exchangeability to freeness with amalgamation we
need some more analytic structure.

\begin{notation}
Consider a $W$-probability space $(\cA,\ff)$, i.e., $\cA$ is a von Neumann algebra and $\ff$ is
a faithful normal state on $\cA$. Consider a sequence of random variables $(x_i)_{i\in\NN}$ in
$\cA$.
\\
1) We denote by $\cA_\infty$ the von Neumann subalgebra generated by $(x_i)_{i \in\NN}$,
and by $\ff_\infty$ the restriction of $\ff$ to $\cA_\infty$.
\\
2) The \emph{tail algebra} of the sequence $(x_i)_{i\in\NN}$ is given by
$$\cA_\tail:=\bigcap_{n=1}^\infty vN(x_k\mid k\geq n),$$
where $vN(x_k\mid k\geq n)\subset\cA$ is the von Neumann algebra generated
by all $x_k$ with $k\geq n$.
\end{notation}
$\cA_\tail$ is a von Neumann subalgebra of $\cA_\infty$ (and thus of $\cA$). In special
cases, $\cA_\tail$ might be trivial, i.e., equal to $\CC 1$, but in general it can be
bigger. In any case, if our sequence is exchangeable, then there exists the unique
$\ff_\infty$-preserving conditional expectation $E\colon \cA_\infty \to \cA_\tail$.
($\ff_\infty$-preserving means of course that $\ff_\infty\circ E= \ff_\infty$.) This is
clear if $\ff$ is a trace. (In this case one does of course not need the exchangeability
and one can introduce $E$ directly as a map from $\cA$ onto $\cA_\tail$.) The general
case, which allows non-tracial states, is treated in \cite{Claus:factorization} and we
adapt a proof from therein for the convenience of the reader.

\begin{proposition}\label{prop:condexp}
Let $(\cA,\varphi)$ be a $W^*$-probability space and suppose the sequence $(x_i)_{i \in
\NN} \subset \cA$ is exchangeable. Then there exists the $\varphi_\infty$-preserving
conditional expectation $E$ from $\cA_\infty$ onto $\cA_\tail$.
\end{proposition}

\begin{proof}
We can assume that $\cA$ is generated by $(x_i)_{i \in \NN}$, i.e., that
$\cA=\cA_\infty$. Now exchangeability implies the stationarity of $(x_i)_{i \in \NN}$ and
thus the existence of an endomorphism $\alpha$ of $\cA$ such that
\[
\varphi \circ \alpha = \varphi \quad \text{and} \quad \alpha(x_i)= \alpha(x_{i+1}).
\]
Let $\cA_I:= vN(x_i | i \in I)$ for $I \subset \NN$ and suppose $a,b \in \bigcup_{|I|<\infty}
\cA_{I}$. Consequently we can assume $a\in \cA_I$ and $b \in \cA_J$ such that there exists $N
\in \NN$ with $I \cap (J +N)=\emptyset$. We infer from exchangeability  that $\varphi(b
\alpha^n(a)) = \varphi(b \alpha^{n+1}(a))$ for all $n \ge N$. Due to minimality this
establishes the limit
\[
\lim_{n\to \infty} \varphi(b \alpha^n(a))
\]
on the weak*-dense *-algebra $\bigcup_{|I|<\infty}\cA_{I}$. A standard approximation argument
ensures now the existence of this limit for $a,b \in \cA$, using the norm density of the
functionals $\{\varphi(b\,\cdot)| b \in \cA\}$ and the boundedness of the set $\{\alpha^n(a)| n
\ge 0\}$. We conclude from this that the pointwise limit of the sequence $(\alpha^n)_{n\in\NN}$
(in the weak operator topology) defines a linear map $Q \colon \cA \to \cA$ such that $Q(\cA)
\subset \cA_\tail$.

It is easily seen that the linear map $Q$ enjoys
\[
\varphi = \varphi \circ Q\quad \text{and} \quad \|Q(a)\| \le \|a\| \text{ for $a\in \cA$.}
\]
Thus $Q$ is a conditional expectation from $\cA$ onto $\cA_\tail$, if we can ensure that $Q(a)
= a$ for all $a \in \cA_\tail$ (see, e.g., \cite{Tak}). To this end let $a \in \cA_\tail$ and
$b \in \bigcup_{|I|<\infty} \cA_I$. We infer from $\cA_\tail\ \subset \alpha^N (\cA)$ and
$\cA_{[N,\infty)} \subset \alpha^N(\cA)$  for all $N \in \NN$ that there exists some $N \in
\NN$ such that $a \in \alpha^N(\cA)$ and $b \in \cA_{[0,N-1]}$. We approximate $a \in \cA$ in
the weak operator topology by a sequence $(a_k)_{k\in\NN} \subset
\bigcup_{|I|<\infty}\alpha^{N}(\cA_{I})$ and conclude further from the definition of $Q$ and
from exchangeability that
\begin{multline*}
\varphi(b Q(a)) = \lim_{k} \varphi(b Q(a_k)) = \lim_k \lim_n \varphi(b \alpha^n(a_k)) = \lim_k
\varphi(b a_k) = \varphi(b a).
\end{multline*}
This shows that $Q(a)= a$ for all $a \in \cA_\tail$. Thus $Q$ is the conditional
expectation of $\cA$ onto $\cA_\tail$ with respect to $\varphi$, which we denote from now
on by $E$.
\end{proof}

Our main goal will be to show that quantum exchangeability implies freeness with respect
to this ($\ff_\infty$-preserving) conditional expectation $E:\cA_\infty \to \cA_\tail$.
Note that, in the non-tracial case, we do \emph{not} define $E$ on $\cA$, but only on
$\cA_\infty \subset \cA$. This is no problem, however, since all our statements on
distribution and freeness with respect to E involve only elements from $\cA_\infty$. So,
in the present section and in Section \ref{section:invariance}, the conditional
expectation $E$ will always be understood as introduced in Proposition
\ref{prop:condexp}. If the reader prefers, she may throughout assume that $\cA$ is
generated by the considered sequence of random variables, i.e., that $\cA = \cA_\infty$
and $\ff=\ff_\infty$.

Let us first check that quantum exchangeability with respect to $\ff$ extends to the same
property with respect to $E$.

\begin{proposition}\label{prop:invarianceE}
Let $(\cA,\ff)$ be a $W^*$-probability space, $(x_i)_{i\in\NN}$ a sequence in $\cA$, and
$E$ the conditional expectation onto the corresponding tail algebra $\cA_\tail$. Assume
that the joint distribution of $(x_i)_{i\in\NN}$ with respect to $\ff$ is invariant
under quantum permutations. Then the same is true for the joint distribution of $(x_i)_{i\in\NN}$
with respect to $E$, i.e., for each $k\in\NN$ and $u=(u_{ij})_{i,j=1}^k$ the generating matrix
of $A_s(k)$, we have
for all $n\in\NN$, all $1\leq i(1),\dots,i(n)\leq k$ and all $b_2,\dots,b_n\in\cAt$ that
\begin{multline}\label{eq:invarianceE}
E[x_{i(1)}b_2 x_{i(2)}\cdots b_{n}x_{i(n)}]\\=\sum_{j(1),\dots,j(n)=1}^k
u_{i(1)j(1)}\cdots u_{i(n)j(n)}\cdot E[x_{j(1)}b_2 x_{j(2)}\cdots  b_{n} x_{j(n)}].
\end{multline}
More generally, for any $p_1,\dots,p_n\in\cAtX$ we have
\begin{multline}\label{eq:invarianceEgeneral}
E[p_1(x_{i(1)}) p_2(x_{i(2)})\cdots p_n(x_{i(n)})]\\=\sum_{j(1),\dots,j(n)=1}^k
u_{i(1)j(1)}\cdots u_{i(n)j(n)}\cdot E[p_1(x_{j(1)}) p_2(x_{j(2)})\cdots  p_n(x_{j(n)})].
\end{multline}
\end{proposition}

\begin{proof}
Fix $n$, $k$, $i(1),\dots,i(n)$. Because of $\ff|_{\cA_\infty}=\ff_\infty=\ff_\infty\circ
E$, \eqref{eq:invarianceE} will follow if we can show that
\begin{multline}\label{eq:toshow}
\ff\bigl(b_1 x_{i(1)}b_2 x_{i(2)}\cdots b_n x_{i(n)}\bigr)\\=\sum_{j(1),\dots,j(n)=1}^k
u_{i(1)j(1)}\cdots u_{i(n)j(n)}\cdot \ff\bigl(b_1 x_{j(1)} b_2 x_{j(2)}\cdots b_n
x_{j(n)}\bigr)
\end{multline}
for all $b_1,\dots,b_n\in\cA_\tail$. This will follow if we can show Equation \eqref{eq:toshow}
for $b_1,\dots,b_n$ of the form $x_{r(1)}\cdots x_{r(p)}$ for $p\in\NN$ and all
$r(1),\dots,r(p)\geq k+1$. (Note that those $b_1, \ldots, b_n$ are not from the tail algebra,
but we can use them to approximate elements from $\cA_\tail$. Indeed, by Kaplansky's theorem,
these approximation can be done on a norm bounded set, where the multiplication of elements is
continuous in the strong operator topology.) Fix such a choice of $b_1,\dots,b_n$ and let $N$
be the maximum of all indices appearing in the product $b_1 x_{i(1)}b_2 x_{i(2)}\cdots b_n
x_{i(n)}$ (written as a product in $x$'s). We extend now the $u$ from $A_s(k)$ to a matrix
$\tilde u=(\tilde u_{ij})_{i,j=1}^N$ according to
$$\tilde u_{ij}=\begin{cases}
u_{ij},& \text{if $1\leq i,j\leq k$}\\
\delta_{ij},&\text{otherwise}
\end{cases}.$$
Then this $\tilde u$ satisfies the defining relations of $A_s(N)$ and the quantum
exchangeability of $x_1,\dots,x_N$ under the action of $\tilde u$ yields exactly Equation
\eqref{eq:toshow}. (Note that a priori we also get factors $\tilde u_{rj}$ corresponding
to the factors $x_r$ of the $b$'s, however, all those will just give a $\delta_{rj}$
contribution, as in this case $r\geq k+1$; thus the $b$'s from the left side of the
equation will just reproduce on the right side of the equation. Also the summation over
the $j(m)$-indices for the $x_{i(m)}$ will a priori be from 1 to $N$, but the factor
$\tilde u_{i(m)j(m)}$ restricts this to the range from 1 to $k$, since $i(m)\leq k$.)

Equation \eqref{eq:invarianceEgeneral} follows from \eqref{eq:invarianceE} by
multilinearity and by checking that we have compatibility of our formulas under
multiplying two $x_i$ together and under inserting a factor $b\in\cA_\tail$. But this is
clear from the relations of the $u_{ij}$; the first compatibility follows from
$$\sum_{j(r),j(r+1)=1}^k u_{ij(r)}u_{ij(r+1)}=\sum_{j(r)=1}^k u_{ij(r)}$$
and the second one from $\sum_{j=1}^k u_{ij}=1$.
\end{proof}

It is clear that the same arguments work also for the case of exchangeability. Since we
will use this version in the proof of Proposition \ref{prop:factorization}, let us state
it here explicitly for later use.

\begin{proposition}\label{prop:invarianceEclassical}
Let $(\cA,\ff)$ be a $W^*$-probability space, $(x_i)_{i\in\NN}$ a
sequence in $\cA$, and $E$ the conditional expectation
onto the corresponding tail algebra $\cA_\tail$. Assume that the joint distribution of
$(x_i)_{i\in\NN}$ with respect to $\ff$ is invariant under classical permutations. Then the
same is true for the joint distribution of $(x_i)_{i\in\NN}$ with respect to $E$, i.e., for
each $k\in\NN$, we have for all $n\in\NN$, all $1\leq i(1),\dots,i(n)\leq k$ and all
$b_2,\dots,b_n\in\cAt$ that 
\begin{equation*}
E[x_{i(1)}b_2 x_{i(2)}\cdots b_{n}x_{i(n)}]=E[x_{\sigma(i(1))}b_2 x_{\sigma(i(2))}\cdots
b_{n} x_{\sigma(i(n))}]
\end{equation*}
for each permutation $\sigma\in S_k$.

More generally, for any $p_1,\dots,p_n\in\cAtX$ we have
\begin{multline}\label{eq:invarianceEgeneralclassical}
E[p_1(x_{i(1)}) p_2(x_{i(2)})\cdots p_n(x_{i(n)})]\\= E[p_1(x_{\sigma(i(1))})
p_2(x_{\sigma(i(2))})\cdots p_n(x_{\sigma(i(n))})]
\end{multline}
for each permutation $\sigma\in S_k$.
\end{proposition}

In the next section we will show how the quantum exchangeability of $E$ will imply
freeness with respect to $E$. For this we will need as an important ingredient the
following factorization property of $E$. This is actually a consequence of the classical
exchangeability property with respect to $\ff$ and was shown in
\cite{Claus:factorization} for more general situations. To establish this desired
factorization property it is crucial to work with an \emph{infinite} sequence of random
variables (see also Remark \ref{rem:finitedefinetti}). In order to make the present paper
self-contained we provide the proof for this factorization in our case. For more details
and generalizations one should see \cite{Claus:factorization}. A related finite version
of that result was also considered in Lemma 2.6 of \cite{AL}.

\begin{proposition}\label{prop:factorization}
Let $(\cA,\ff)$ be a $W^*$-probability space and $(x_i)_{i\in\NN}$ a sequence in $\cA$
whose joint distribution is invariant under classical permutations. Let $E$ be the
conditional expectation onto the tail algebra $\cA_\tail$ of the sequence (see
Proposition \ref{prop:condexp}). Then $E$ has the following factorization property: for
all $n\in\NN$, all polynomials $p_1,\dots,p_n\in\cAtX$ and all $i(1),\dots,i(n)\in\NN$ we
have
\begin{multline}\label{eq:factorization}
E\bigl[p_1(x_{i(1)}) \cdots  p_l(x_{i(l)})\cdot \cdots
p_n(x_{i(n)})\bigr]\\=E\bigl[p_1(x_{i(1)}) \cdots  E[p_l(x_{i(l)})]\cdots
p_n(x_{i(n)})\bigr]
\end{multline}
whenever $i(l)$ is different from all the other $i(r)$ ($r\not=l$).
\end{proposition}

Note that exchangeability with respect to $\ff$ does not imply a factorization property
for $\ff$, but only for the conditional expectation $E$. This is of course responsible
for the fact that we get freeness with respect to $E$ and not with respect to $\ff$ in
our noncommutative de Finetti theorem.

\begin{proof}
Without restriction we will assume in the following that $\cA$ is generated by our
sequence, i.e., $\cA=\cA_\infty=vN(x_i\mid i\in\NN)$. By $L^2(\cA,\ff)$ we will denote
the GNS Hilbert space corresponding to $\ff$, equipped with the inner product $\la a,b\ra
=\ff(a^*b)$.

The exchangeability of our sequence implies then that we can define on $L^2(\cA,\ff)$ an
isometric shift $\alpha$ given by
$$\alpha(x_{i(1)}\cdots x_{i(n)})=x_{i(1)+1}\cdots x_{i(n)+1}.$$
Restricted to $\cA\subset L^2(\cA,\ff)$, this shift maps $\cA$ into itself and acts there
as an endomorphism. Let us denote the fixed point algebra of this shift by
$$\cA_\alpha:=\{a\in\cA\mid \alpha(a)=a\}.$$
Clearly, $\cA_\alpha\subset \cA_\tail$. We want to show that also $\cA_\tail \subset
\cA_{\alpha}$. For this, fix $b \in \cA_\tail$ and consider, for $m\in\NN$ and
$r(1),\dots,r(m)\in\NN$, the moment $\ff(x_{r(1)}\cdots x_{r(m)}b)$. By approximating $b$ with
noncommutative polynomials in $\{x_k\mid k>\max(r(1),\dots,r(m))\}$ and using the
exchangeability of the $(x_i)_{i\in\NN}$, one sees that
$$\ff(x_{r(1)}\cdots x_{r(m)}b)=\ff(x_{r(1)}\cdots x_{r(m)} \alpha(b))$$
for all $m\in\NN$, $r(1),\dots,r(m)\in\NN$. But then one also has
$$\ff(a b)=\ff(a \alpha(b))$$
for all $a \in \cA$ and hence $b = \alpha(b)$. Since this is true for any $b \in \cA_\tail$ we
actually have that $\cA_{\alpha}= \cA_\tail$. Thus Proposition $\ref{prop:condexp}$ entails
that the $\varphi$-preserving conditional expectation $E_\alpha$ from $\cA$ onto $\cA_{\alpha}$
exists and equals the conditional expectation $E$ onto the tail algebra.

We recall next that the mean ergodic theorem of von Neumann implies that we have for all
all $a\in\cA$
$$\lim_{m\to\infty}\frac 1m \sum_{i=1}^m \alpha^i(a)= E_\alpha[a]=E[a],$$
in the strong operator topology (see, e.g., \cite{Claus:factorization}).

Now let us consider the situation as in our proposition. By Proposition
\ref{prop:invarianceEclassical} we have exchangeability of $E$ according to
\eqref{eq:invarianceEgeneralclassical}; this means that we have in our situation
$$E\bigl[p_1(x_{i(1)}) \cdots  p_l(x_{i(l)}) \cdots
p_n(x_{i(n)})\bigr]=E\bigl[p_1(x_{i(1)}) \cdots  p_l(x_i) \cdots p_n(x_{i(n)})\bigr]$$
for any $i>N:=\max\{i(1),\dots,i(n)\}.$ But then we also have
\begin{align*}
E\bigl[p_1(x_{i(1)}) &\cdots  p_l(x_{i(l)}) \cdots p_n(x_{i(n)})\bigr]\\&=\frac
1m\sum_{i=N+1}^{N+m} E\bigl[p_1(x_{i(1)}) \cdots  p_l(x_i) \cdots p_n(x_{i(n)})\bigr]\\
&= E\left[p_1(x_{i(1)}) \cdots  \left(\frac 1m\sum_{i=N+1}^{N+m}p_l(x_i)\right) \cdots
p_n(x_{i(n)})\right].
\end{align*}
But
$$ \frac 1m\sum_{i=N+1}^{N+m}p_l(x_i)=\frac 1m \sum_{i=1}^{m}
\alpha^i(p_l(x_N))$$ converges by the mean ergodic theorem to
$E[p_l(x_N)]=E[p_l(x_{i(l)})]$ and thus we get
\begin{multline*}
E\bigl[p_1(x_{i(1)}) \cdots  p_l(x_{i(l)})\cdot \cdots
p_n(x_{i(n)})\bigr]\\=E\bigl[p_1(x_{i(1)}) \cdots  E[p_l(x_{i(l)})]\cdots
p_n(x_{i(n)})\bigr]
\end{multline*}
Note again for the convergence argument that multiplication on norm bounded sets is continuous
in the strong operator topology.
\end{proof}

\section{Invariance under quantum permutations implies\\ freeness over the tail algebra} \label{section:invariance}

We will now provide the proof of the implication `$\ref{item:deFinetti-a}\implies \ref{item:deFinetti-b}$' of Theorem
\ref{thm:deFinetti}. Throughout this section $E$ will denote the conditional expectation
as introduced in Proposition \ref{prop:condexp}.

Let us first address the identical distribution of the $x_i$ with respect to E. For this
we actually need only the classical exchangeability.

\begin{proposition}
Let $(\cA,\ff)$ be a $W^*$-probability space and consider a sequence $(x_i)_{i\in\NN}$ in
$\cA$. Assume that the joint distribution of $(x_i)_{i\in\NN}$ with respect to $\ff$ is
invariant under classical permutations. Then the sequence $(x_i)_{i\in\NN}$ is identically
distributed with respect to the conditional expectation $E$ onto the tail algebra of
$(x_i)_{i\in\NN}$.
\end{proposition}

\begin{proof}
This is just a special case of Proposition \ref{prop:invarianceEclassical}.
\end{proof}

Now we will address the freeness property. For this one needs, as in the classical case,
an \emph{infinite} sequence of random variables. One should, however, note that the only
way in which this infinity enters is via the factorization property of Proposition
\ref{prop:factorization} (which relied in the end on the mean ergodic theorem). If this
factorization property is assumed then it is feasible that a more elaborated version of
the following arguments is also applicable to finite sequences of random variables.

We will check that $x_1,x_2,\dots$ are free with respect to $E$ by verifying the defining
relations for freeness. So let us consider $n\in\NN$ and polynomials $p_1,\dots,p_n\in\cAtX$
such that $E[p_i(x_1)]=0$ for all $i=1,\dots,n$. Then we have to show that for all
$i(1)\not=i(2)\not=\dots\not=i(n)$
$$E[p_1(x_{i(1)})\cdots p_n(x_{i(n)})]=0.$$
We will prove this (for fixed $n$ and $p_1,\dots,p_n$) by induction over the number of
blocks of $\ker \ii$, starting from the biggest number and going down in steps of one. To
get started consider the biggest number of blocks, which is $n$. Then all
$i(1),\dots,i(n)$ are different and, by an iterated application of the factorization
property, Proposition \ref{prop:factorization}, we have
$$E[p_1(x_{i(1)})\cdots p_n(x_{i(n)})]=E[p_1(x_{i(1)})]\cdots E[p_n(x_{i(n)})]=0.$$
Now assume, for some $r$, we have proved that $E[p_1(x_{i(1)})\cdots p_n(x_{i(n)})]=0$
whenever $i(1)\not=i(2)\not=\dots\not=i(n)$ and $\ker \ii$ has at least $r+1$ blocks. We
want to show the same for the case that $\ker \ii$ has $r$ blocks.

By Proposition \ref{prop:invarianceE}, we have
\begin{align}\label{eq:inductionstep}
&E[p_1(x_{i(1)})\cdots p_n(x_{i(n)})]\\&=\sum_{j(1),\dots,j(n)=1}^k u_{i(1)j(1)}\cdots
u_{i(n)j(n)}\cdot E[p_1(x_{j(1)})\cdots p_n(x_{j(n)})]\notag\\
&= \sum_{\pi\in\cP(n)}\sum_{\substack{{j(1),\dots,j(n)=1}\\{\ker \jj\;=\pi}}}^k u_{i(1)j(1)}\cdots
u_{i(n)j(n)}\cdot E[p_1(x_{j(1)})\cdots p_n(x_{j(n)})]\notag
\end{align}
Let us first observe that $\jj$'s where two neighboring indices are the same do not
contribute; this follows from the fact that $u_{i(s)j(s)}u_{i(s+1)j(s+1)}=0$ if
$j(s)=j(s+1)$ because $i(s)\not=i(s+1)$. Thus we only have to sum over
$\jj=(j(1),\dots,j(n)$ in the above sum for which $j(1)\not=j(2)\not=\dots\not=j(n)$. But
for those our induction hypothesis applies and thus we see that in the summation over
$\pi\in\cP(n)$ we can restrict to $\pi$ which have at most $r$ blocks. Since
$E[p_1(x_{i(1)})\cdots p_n(x_{i(n)})]$ is invariant under permutations, we can fix
specific (different) $i$-values for the $r$ blocks of $\ker \ii$; let us take
$1,3,5,\dots, 2r-1$ for them.

Let us now choose $k=2r$ and a specific $u=(u_{ij})_{i,j=1}^{2r}$, namely
\begin{equation}\label{eq:matrixu}
u=\begin{pmatrix}
q_1 & 1-q_1& 0&0&\hdots&0&0\\
1-q_1&q_1&0&0&\hdots&0&0\\
0&0&q_2&1-q_2&\hdots&0&0\\
0&0&1-q_2&q_2&\hdots&0&0\\
\vdots&\vdots&\vdots&\vdots&\ddots&\vdots&\vdots\\
0&0&0&0&\hdots &q_r&1-q_r\\
0&0&0&0&\hdots &1-q_r&q_r
\end{pmatrix},
\end{equation}
where $q_1,\dots,q_r$ are arbitrary projections. With this choice of $u$ we have that for
a non-vanishing $u_{ij}$ the $j$-value determines the $i$-value (since we only have the
odd numbers as possible $i$-values), i.e., we have $\ker \jj\leq \ker \ii$; thus in the
sum \eqref{eq:inductionstep} we can restrict to $\pi\leq \ker \ii$. But since we also
restricted to $\pi$ with at most $r$ blocks, we are just left with the one possibility
$\pi=\ker \ii$, i.e., with the above $u$ we can continue \eqref{eq:inductionstep} as
follows:
\begin{align*}
E[p_1(x_{i(1)})&\cdots p_n(x_{i(n)})]\\
&= \sum_{\substack{{j(1),\dots,j(n)=1}\\{\ker \jj\;=\ker \ii}}}^{2r} \hspace{-12pt} u_{i(1)j(1)}\cdots
u_{i(n)j(n)}\cdot E[p_1(x_{j(1)})\cdots p_n(x_{j(n)})]\\
&= \left(\sum_{\substack{{j(1),\dots,j(n)=1}\\{\ker \jj\;=\ker \ii}}}^{2r} \hspace{-12pt} u_{i(1)j(1)}\cdots
u_{i(n)j(n)}\right)\cdot E[p_1(x_{i(1)})\cdots p_n(x_{i(n)})]
\end{align*}
In the last step we have used the fact that, because of the identical distribution of the
$x_i$'s with respect to $E$, the term $E[p_1(x_{j(1)})\cdots p_n(x_{j(n)})]$ depends only
on $\ker \jj$.

If we can show that
\begin{align}\label{eq:bigsum}
\sum_{\substack{{j(1),\dots,j(n)=1}\\{\ker \jj\;=\ker \ii}}}^{2r} \hspace{-12pt} u_{i(1)j(1)}\cdots
u_{i(n)j(n)}
\end{align}
is different from 1, then this implies that $E[p_1(x_{i(1)})\cdots p_n(x_{i(n)})]$ has to
vanish, and we are done.

Note that if $\ker \ii$ is non-crossing then the sum \eqref{eq:bigsum} is actually equal to 1
for any $u$ satisfying the relations of the quantum permutation group. However, if $\ker \ii$
is non-crossing then the condition $i(1)\not=i(2)\not=\dots\not=i(n)$ implies that $\ker \ii$
must have at least one singleton, i.e., one i-index appears only once and then the
factorization property \eqref{eq:factorization} gives right away that $E[p_1(x_{i(1)})\cdots
p_n(x_{i(n)})]=0$. Thus we can restrict to crossing $\ker \ii$ when considering
\eqref{eq:bigsum}.

Note also that if all the $u_{ij}$ in \eqref{eq:bigsum} commute, then we will actually
get 1 (independent of whether $\ker \ii$ is crossing or non-crossing); this shows that
invariance under usual permutations is (clearly) not strong enough to imply freeness. We
have to invoke some real quantum permutations, i.e., we should choose the $q_1,\dots,q_r$
in \eqref{eq:matrixu} as non-commuting. However, it suffices to take just two of them as
non-commuting. Since $\ker \ii$ is crossing we can choose two blocks which have a
crossing. For those two blocks we choose some non-commuting projections $p$ and $q$,
whereas for all the other blocks we choose their projections as 1. Then the sum
\eqref{eq:bigsum} reduces to
$$(pq)^s+(p(1-q))^s+((1-p)q)^s+((1-p)(1-q))^s$$
or to
$$(pq)^sp+(p(1-q))^sp+((1-p)q)^s(1-p)+((1-p)(1-q))^s(1-p)$$
for some $s\geq 2$. It is clear that this is not equal to 1 for generic projections $p$
and $q$. Actually it is fairly easy to see that these expressions are equal to 1 if and
only if $p$ and $q$ commute.

This finishes the proof of the implication `$\ref{item:deFinetti-a}\implies \ref{item:deFinetti-b}$' 
in Theorem \ref{thm:deFinetti}.

\begin{remark}\label{rem:finitedefinetti}
As already mentioned before, our noncommutative de Finetti theorem is not true for a
finite number of random variables. To infer freeness from quantum exchangeability one
needs, as in the classical case, infinitely many variables. To prove that claim we will
in the following present an example, which can be considered as an analogue to classical
urn models without replacement.

Consider the quantum permutation group $A_s(n)$ itself, with defining matrix
$u=(u_{ij})_{i,j=1}^n$. Then, by a fundamental result of Woronowicz \cite{Wor}, there
exists a normalized Haar functional $\psi:A_s(n)\to\CC$. Consider now the $GNS$
representation of $A_s(n)$ with respect to $\psi$; this gives a $W^*$-probability space
$(\cA,\psi)$, where $\cA$ is the weak closure of $A_s(n)$. The defining invariance
property of the Haar functional implies that each column of $u=(u_{ij})_{ij=1}^n$ is
invariant under quantum permutations from $A_s(n)$. To be concrete, let us consider the
first column, $u_{11},\dots,u_{n1}$. These $n$ elements are quantum exchangeable with
respect to $\psi$. However, we claim that there does not exist a conditional expectation
$E$ from $\cA$ onto a von Neumann subalgebra $\cB\subset\cA$ with $\psi\circ E= \psi$,
such that $u_{11},\dots,u_{n1}$ are identically distributed and free with respect to $E$.
Assume the contrary. Since $u_{11}u_{21}=0$, we would have
$$0=E[u_{11} u_{21}]=E[u_{11}] E[u_{21}]=E[u_{11}] E[u_{11}].$$
Since $E[u_{11}]$ is selfadjoint this implies that $E[u_{11}]=0$, and thus also
$$\psi(u_{11})=\psi\bigl(E[u_{11}]\bigr)=0.$$
However, $u_{11}=u_{11}u_{11}^*$ and $\psi$ is faithful, thus $u_{11}=0$, which is a
contradiction.
\end{remark}

\section*{Acknowledgement}
We thank Franz Lehner for some helpful comments on an earlier version of the manuscript.

\end{document}